\documentclass[12pt,article]{amsart}
\usepackage{color}
\usepackage{latexsym}
\usepackage{amsmath,amssymb,dsfont,amsfonts}
\usepackage{mathrsfs}
\usepackage{verbatim}
\usepackage{graphicx}
\usepackage{graphics}
\usepackage{geometry}
\usepackage{booktabs}
\usepackage[shortlabels]{enumitem}
\usepackage{xcolor}
\usepackage{fancyhdr}
\newtheorem{remark}{Remark}[section]
\newtheorem{proposition}[remark]{Proposition}
\newtheorem{theorem}[remark]{Theorem}
\newtheorem{definition}[remark]{Definition}
\newtheorem{corollary}[remark]{Corollary}
\newtheorem{lemma}[remark]{Lemma}
\newtheorem{example}[remark]{Example}

\def\R{\mathbb R}
\def\M{\mathbb M}
\def\N{\mathbb N}

\def\X{\mathbb X}

\def\shc{{\mathcal C}}
\def\shd{{\mathcal D}}
\def\shf{{\mathcal F}}

\def\shl{{\mathcal L}}

\def\shr{{\mathcal R}}

\newenvironment{prooff}{{\bf \textit{Proof}}}{\hfill $\Box$ \\}

\numberwithin{equation}{section}

\allowdisplaybreaks

\title{Rough paths and regularization}

\author{Andr\'e  O. GOMES $^1$}

\author{Alberto OHASHI$^2$}

\author{Francesco RUSSO$^3$}

\author{Alan TEIXEIRA$^4$}

\address{$1$ ENSTA Paris, Institut Polytechnique de Paris,
 Unit\'e de Math\'ematiques appliqu\'ees, 828, boulevard des Mar\'echaux, F-91120 Palaiseau, France.
 Departamento de Matem\'atica Universidade Estadual de Campinas 13081-970 Campinas SP-Brazil.} \email{andredeoliveiragomes2@gmail.com}

\address{$2$ Departamento de Matem\'atica, Universidade de Bras\'ilia, 70910-900, Bras\'ilia, Brazil.} \email{amfohashi@gmail.com}

\address{$3$ ENSTA Paris, Institut Polytechnique de Paris,
 Unit\'e de Math\'ematiques appliqu\'ees, 828, boulevard des Mar\'echaux, F-91120 Palaiseau, France}
 \email{francesco.russo@ensta-paris.fr}
\address{$4$ ENSTA Paris, Institut Polytechnique de Paris,
 Unit\'e de Math\'ematiques appliqu\'ees, 828, boulevard des Mar\'echaux, F-91120 Palaiseau, France}
\email{alan.teixeira@ensta-paris.fr}

\date{June 14th 2021}

\fancypagestyle{normal}{
	\fancyhead[R]{\thepage} 
}

\begin{document}

\begin{abstract}
Calculus via regularizations and rough paths are two methods
to approach stochastic integration and calculus close to pathwise
calculus. The origin of rough paths theory is purely deterministic,
calculus via regularization is based on deterministic techniques
but there is still a probability in the background.
The goal of this paper is to establish a connection between stochastically
controlled-type processes, a concept reminiscent from rough paths theory,
and the so-called weak Dirichlet processes. As a by-product, we present
the connection between rough and Stratonovich integrals for càdlàg weak
Dirichlet processes integrands and continuous semimartingales integrators.

\end{abstract}

\maketitle



\section{Introduction}

This paper focuses on two variants of stochastic calculus
of pathwise type: calculus via regularization
and rough paths. The recent literature on rough paths
is very rich and it is impossible to list it here completely.
It was started in \cite{lyons} continued by the monograph
\cite{lyonsq} which focused on
rough differential equations. The corresponding
 integral was introduced
later by M.  Gubinelli, see \cite{gubinelli2004}.
Later, a great variety of contributions on
the subject appeared and it is not possible to list all of them.
We refer however to the  monograph \cite{hairerbook}
to a fairly rich list of references and for a complete
development of the subject.
In spite of some recent work mixing probability and deterministic theory,
see e.g. \cite{le2020stochastic,chevyrev2019,friz2020rough}, the theory of rough paths is essentially deterministic.

Stochastic calculus via regularization was started first
by F. Russo and P. Vallois in \cite{RVCras}.
The calculus was later continued in \cite{rv1,  rv96, rv4}
in the framework of continuous integrators, essentially
with finite quadratic variation. The case
of processes with higher variation was first introduced in
\cite{er, er2} and continued in \cite{coviello1, grv, gnrv, gradno, russoviens,
bbv11}, especially in relation with fractional Brownian motion and related
processes. A not very recent survey paper in  the  framework of
finite dimensional processes is \cite{russo2007elements}. Stochastic
calculus via regularization for processes taking values in Banach spaces, with
applications to the path-dependent case, was realized in
\cite{DGR, DGRnote}  and  in \cite{cosso_russo15a}.
The case of real-valued jump integrators was first introduced in \cite{rv93}
and then deeply investigated in \cite{nunno} and later
by \cite{BandiniRusso1}.
Applications to mathematical finance (resp. to fluidodynamics modeling)
were published in \cite{coviello2bis} (resp. \cite{flaguru}).

An important notion which emerged in calculus via regularization
is the notion of weak Dirichlet processes, started in \cite{er, gr}.
Such a process $X$ is the sum of a local martingale $M$ and an orthogonal
process $A$ such that $[A,N] = 0$ for any continuous martingale.
This constitutes a natural generalization of
the notion of semimartingale and of Dirichlet process (in the sense
of F\"ollmer), see \cite{FolDir}. In particular, \cite{gr}
allowed to establish chain rule type decomposition extending It\^o
formulae with applications to control theory, see \cite{gr1}.
That concept was extended to the jump case by \cite{cjms}
and its related calculus was performed by \cite{BandiniRusso1}
with applications to BSDEs, see \cite{BandiniRusso2}.
In \cite{DGR1, DGR2} one has performed weak
Dirichlet decomposition of real functional of Banach space-valued
processes. In \cite{cosso_russo15a, DGRclassical} one has investigated
strict solutions of path-dependent PDEs.

In this paper
 we  wish first to give a key to revisit 
the theory of rough paths under the perspective
of stochastic calculus via regularizations.
The idea here is not to summarize the theory
of rough paths integrals, but to propose a variant version which
is directly probabilistic.
In particular, we emphasize
the strong link between the notion of weak Dirichlet
process and one of stochastically controlled process,
which is a stochastic
version of the one proposed by Gubinelli \cite{gubinelli2004}.
According to Definition \ref{DStochGub} such a process fulfills
\begin{equation}\label{stintr}
Y_t - Y_s = Y'_s (X_t - X_s) + R^Y_{s ,t}, s < t,
\end{equation}
where
\begin{equation}\label{orthintr}
\lim_{\varepsilon\rightarrow 0^+}\frac{1}{\varepsilon}\int_0^t R^Y_{s,s+\varepsilon} (X_{s+\varepsilon} - X_s)ds=0,
\end{equation}
in probability for each $t \in [0,T]$.

  Here, $X$ is the reference driving noise, $Y'$ is a process (not necessarily admitting $\gamma$-H\"older continuous paths). The orthogonality condition (\ref{orthintr}) resembles the $2\gamma$-H\"older-regularity condition reminiscent from \cite{gubinelli2004}.

Propositions \ref{stochex1} and \ref{stochex2} present the connection between weak Dirichlet processes and stochastically controlled processes. In particular, when the reference driving noise is a martingale, then both concepts coincide. As a side effect, Theorem \ref{RoughSem} shows Stratonovich integration as a stochastic rough-type integration for weak Dirichlet integrands and continuous semimartingale integrators. The connection between rough paths theory with semimartingales has been investigated by some authors. \cite{coutin2005semi} shows pathwise Wong-Zakai-type theorems for Stratonovich SDEs driven by continuous semimartingales. In particular, the
  integral defined by rough paths theory agrees with Stratonovich integrals
  for real-valued functions $f(X)$ of the driving noise $X$, see also Proposition 17.1 in \cite{friz}. Recently,  \cite{friz2020rough} introduces a concept of rough semimartingales and develops the corresponding stochastic integration having a deterministic rough path in the background and mixing with $p$-variation regularity. Beyond
  semimartingale driving noises, we drive attention to the recent work of \cite{tindelliu}. The authors have established the connection between rough integrals and trapezoidal Riemann sum approximations for controlled processes integrands (in the pathwise sense of \cite{gubinelli2004}) and a general class of Gaussian driving noises.

  In this article, we take full advantage of the probability measure and the stochastic controllability (\ref{stintr}) to establish consistency between stochastic rough-type and Stratonovich integrals  for more general integrands.
  In the companion paper in preparation \cite{ORrough}, a detailed analysis on stochastic rough-type integrals driven by Gaussian rough paths and their connection with Stratonovich and Skorohod integrals is presented.

The paper is organized as follows.
After this introduction, in  Section \ref{S3} we introduce
some notations about matrix-valued calculus via regularization.
In Section \ref{S4} we introduce the notion of stochastically
controlled paths and the one of stochastic Gubinelli derivative,
under the inspiration of the classical rough paths theory.
We link this with the notion of Dirichlet process.
In Section \ref{SSecond} we introduce the {\bf second order process}
(connected with the L\'evy area) and
finally in Section \ref{S74} discuss the notion of rough stochastic
integrals via regularization, examining carefully the case when
the integrator is a semimartingale.

\section{Preliminary notions}

\subsection{Basic notations}

We introduce here some basic notations intervening in the
paper. $T >0$ will be a finite fixed horizon.
Regarding linear algebra, vectors or elements of $\R^d$
will be assimilated to column vectors, so that if $x$ is a vector in $\R^d$, then $x^\top$ is a row vector.

We continue fixing some  notations. In the sequel,
 finite-dimensional Banach spaces $E$ will be equipped with a norm $|\cdot|$,
typically $E = \R^d$.
Let $T > 0$ be a  fixed maturity.
For $\alpha \in ]0,1]$, the   notation $C^{[\alpha]}([0,T];E)$ is reserved for
$E$-valued paths  defined on $[0,T]$, H\"older continuous of index
$\alpha \in ]0,1]$.
For $X \in {C}^{[\alpha]}([0,T];E)$, the usual seminorm is given by
$$\|X\|_{\alpha}:=\sup_{s,t\in [0,T], s \neq t}\frac{|X_{s,t}|}{|t-s|^\alpha},$$
where we set
\begin{equation} \label{EIncr}
  X_{s,t}:= X_t-X_s, \ 0 \le s,t \le T.
\end{equation}
When $E = \R$ we simply write $C^{[\alpha]}([0,T])$

 For a two-parameter function
$R:[0,T]^2 \rightarrow \R$, vanishing on the diagonal
$ \{(s,t) \vert 0 \le s = t \le T\}$, we write
$R(s,t):= R_{s,t}$.
We say that $ R \in C^{[\alpha]}([0,T]^2)$ if

\begin{equation}\label{EIncrBis}
  \|R\|_{\alpha}:=\sup_{s,t\in [0,T]^2}\frac{|R_{s,t}|}{|t-s|^\alpha} < \infty.
  \end{equation}
By convention the quotient $\frac{0}{0}$ will set to zero. In the sequel, if $ n \in \N^*$, we will extend a function
$ R \in C([0,T]^n)$ to $\R^n$ by continuity,
setting
\begin{equation} \label{EContExtension}
 R_{t_1, \ldots, t_n}:= R_{(t_1 \wedge T), \ldots, (t_n \wedge T)}.
 \end{equation}

$(\Omega,\shf,P)$ will be a fixed probability space.
Let $X^1, X^2$ be two stochastic processes, continuous for simplicity.

We  introduce
\begin{equation}\label{SIVR1Cn}
C(\varepsilon , X^1, X^2)(t)=\int _0^t
\frac{\big(X^1_{s+\varepsilon}-X^1_s\big)\big(X^2_{s+\varepsilon}
-X^2_s\big)}{\varepsilon}ds, \quad t \ge 0.
\end{equation}
In the sequel $(\shf_t)$ will be a filtration fulfilling the usual condition.
\begin{definition} \label{D24}
\begin{enumerate}
\item
The {\bf covariation}  of $X^1$ and $X^2$
is the continuous process (whenever it exists) $[X^1,X^2]$ such that, for $t\geq 0$,
$$C(\varepsilon , X^1, X^2)(t)\ \mbox{ converges in probability  to } \  [X^1,X^2]_t.$$
We say that the covariation $[X^1, X^2]$ exists in the strong sense
if moreover
 \begin{equation}\label{e2}
   \sup_{0<\varepsilon\leq
     1}
\int _0^T
  \frac{\left \vert \big(X^1_{s+\varepsilon}-X^1_s\big)\big(X^2_{s+\varepsilon}
-X^2_s\big) \right\vert}{\varepsilon}ds < \infty.
\end{equation}

\item A vector of processes $(X^1,\cdots,X^d)^\top$ is said to have all its {\bf mutual  covariations}
 if
 $[X^i,X^j]$ exists for every $1 \le i,j \le d$.
\item A real process $X$ is said to be {\bf strong finite cubic variation process},
  see \cite{er2}, if there is a process $\xi$ such that,
  for every $t \in [0,T]$,
  $$
  \int _0^t \frac{\vert X^1_{s+\varepsilon}-X^1_s\big\vert^3}
    {\varepsilon}ds \rightarrow \xi,
$$
in probability. If $\xi = 0$ then $X$ is said to have
 zero cubic variation.
\item
  A real-valued (continuous) $(\shf_t)$-{\bf martingale
    orthogonal process} $A$ is a continuous adapted process
 such that $[A,N] = 0$ for every
 $(\shf_t)$-local martingale $N$.
 A  real-valued (continuous) $(\shf)$-{\bf weak Dirichlet process} is
 the sum of a continuous $(\shf_t)$-local martingale $M$
 and an  $(\shf_t)$-martingale
    orthogonal process.
\end{enumerate}
\end{definition}

\begin{remark} \label{RBracketSem}
\begin{enumerate}
\item If $X^1, X^2$ are two semimartingales
  then $(X^1,X^2)^\top$ has all its mutual covariations, see Proposition 1.1 of
  \cite{rv2} and $[X^1,X^2]$ is the classical covariation
  of semimartingales.
\item It may happen that $[X^1,X^2]$ exists but $(X^1,X^2)^\top$ does
  not have all its mutual covariations, see Remark 22 of
  \cite{russo2007elements}.
\item If $X^1$ (resp. $X^2$) has $\alpha$-H\"older
  (resp.  $\beta$-H\"older) paths with
  $\alpha + \beta > 1$, then $[X^1,X^2]=0$,
  see Propositions and 1 of  \cite{russo2007elements}.
  \end{enumerate}
 \end{remark}

Suppose that $M = (M^1,\ldots,M^d)$,
and $M^1, \ldots,M^d$ are real-valued local martingales.
In particular $M^\top$ is an $\R^d$-valued local martingale.
We  denote by  $\shl^2(d[M,M])$ the space
of processes $H = (H^1,\ldots,H^d)$ where $H^1,\ldots,H^d$
are real progressively measurable processes
and
\begin{equation} \label{SIVR32progVect}
\sum_{i,j} \int_0^T  H^i_s H^j_s d[M^i,M^j]_s < \infty \quad {\rm  a.s.}
\end{equation}
$\shl^2(d[M,M])$ is an $F$-space with respect to the metrizable
topology $d_2$ defined as follows:
  $(H^n)$ converges to $H$ when $n \rightarrow \infty$ if $$ \sum_{i,j} \int_0^T ((H^n)^i_s - H^i_s)  ((H^n)^j_s - H^j_s) d[M^i,M^j]_s
 \rightarrow 0,
 $$
 in probability,
 when $n \rightarrow \infty$.

\bigskip

Similarly as in (27), in Section 4.1 of \cite{russo2007elements}, one can prove the following.

.

\begin{proposition}\label{RC2}
Let $X^1, X^2$ be two processes such that
$(X^1,X^2)^\top$ has all its mutual covariations, and
$H$ be a continuous (excepted eventually on a countable
number of points) real-valued process,  then
$$ \frac{1}{\varepsilon} \int_0^\cdot H_s (X^1_{s+\varepsilon} - X^1_s)
 (X^2_{s+\varepsilon} - X^2_s) ds
\rightarrow \int_0^\cdot H_s d[X^1,X^2]_s$$
    in the ucp sense, when $\varepsilon \rightarrow 0$.
\end{proposition}

 \subsection{Matrix-valued integrals
   via regularization}

\label{S3}

Here we will shortly discuss about matrix-valued stochastic
integrals via regularizations.
Let $\M^{n\times d}$ be the linear space of the
real $n \times d$ matrices, which in the rough paths literature
are often associated with tensors.

For every $(s,t) \in \Delta := \{(s,t) \vert 0 \le s \le t \le T\},$
we introduce
two $\M^{n \times d}$-valued  stochastic integrals via regularizations.
Let $X$ (resp. $Y$) be an  $\R^d$-valued
(resp. $\R^n$-valued) continuous process (resp. locally integrable process)
 indexed by $[0,T]$.

So $X = (X^1, \ldots, X^d)^\top$ (resp. $Y = (Y^1, \ldots, Y^n)^\top$).

\begin{equation}\label{SIVR1Matrix}
  \int_s^t Y \otimes d^-X:=
  \lim_{\varepsilon \rightarrow 0+} \int _s^t Y_r
\frac{(X_{r+\varepsilon}-X_r)^\top}{\varepsilon} dr,
\end{equation}
\begin{equation}\label{SIVR1MatrixSymm}
\left( {\rm resp.} \ \int_s^t Y \otimes d^\circ X:=
  \lim_{\varepsilon \rightarrow 0+} \int _s^t \frac{Y_r  + Y_{r+\varepsilon}}{2}
\frac{(X_{r+\varepsilon}-X_r)^\top}{\varepsilon} dr \right),
\end{equation}
provided that previous limit holds in probability
and the random function $t \mapsto \int_0^t Y \otimes d^-X,$
(resp. $t \mapsto \int_0^t Y \otimes d^\circ X$),
 admits a continuous version.
In particular
$$ \left(\int_s^t Y \otimes d^-X\right)({i,j}) = \int_s^t Y^i \otimes d^-X^j.$$
We remark that
$\int_s^t Y \otimes d^-X$ exists if and only if
$\int_s^t Y^i \otimes d^-X^j$ exist for every $1 \le i \le n, 1 \le j \le d$. \\

Suppose now that $Y$ is continuous.
We denote by $[X,Y]$ the matrix
$$ [X,Y](i,j) = [X^i,Y^j], 1 \le i \le d, 1 \le j \le n,
 $$
provided those covariations exist.
If $n = d$ and $X = Y$, previous matrix exists for instance
if and only if $X$ has all its mutual covariations.

We will denote by $[X,X]^\R$ the {\bf scalar
quadratic variation} defined as
the real continuous process
(if it exists) such that
$$ [X,X]^\R_t := [X^\top, X]=  \lim_{\varepsilon \rightarrow 0}
\int _0^t
\frac{\vert X_{s+\varepsilon}-X_s\vert^2}
{\varepsilon}ds, \quad t \in [0,T],
 $$
when the limit holds in probability.
 $[X,X]^\R$, when it exists, is an increasing process.
When $X^i$ are finite quadratic variation processes
for every $1 \le i \le d$,
then
$$ [X,X]^\R = \sum_{i=1}^d [X^i,X^i].$$

We recall that $\R^d$-valued continuous process is called
semimartingale
with respect to a filtration $(\shf_t)$,
if all its components are semimartingales.

\section{Stochastically controlled paths and Gubinelli derivative}

\label{S4}

In \cite{gubinelli2004}, the author introduced
a class of
controlled paths $Y$ by a reference function.

\begin{definition}\label{Gubexample} (Gubinelli).
  Let $X$ be a function belonging to $C^{[\gamma]}([0,T];E)$ with
 $\frac{1}{3} < \gamma
  < \frac{1}{2}$. An element $Y$ of $C^{[\gamma]}([0,T])$
  is called {\bf weakly controlled} (by $X$)
if there exists a function $Y' \in {C}^{[\gamma]}([0,T];E)$
(here by convention, $Y'$ will be a row vector),
so that the remainder term $R$ defined by the relation
$$Y_{s,t} = Y'_s (X_{t} - X_s) + R_{s, t}, s,t \in [0,T],$$
belongs to ${C}^{[2\gamma]}([0,T]^2).$

\end{definition}

From now on $X$ will stand for a fixed $\R^d$-valued reference
 continuous process.
The definition below is inspired by previous one.
\begin{definition}  \label{DStochGub}
\begin{enumerate}
\item
We say that an
  $\R$-valued stochastic process $Y$ is \textbf{stochastically
 controlled}
  by $X$ if there exists an
 $\R^d$-valued stochastic process $Y'$ (here again
indicated by a row vector)
 so that the remainder term $R^Y$
defined by the relation
\begin{equation}\label{lineardep}
Y_t - Y_s = Y'_s (X_t - X_s) + R^Y_{s ,t}, s < t,
\end{equation}
satisfies

\begin{equation}\label{orth}
\lim_{\varepsilon\rightarrow 0^+}\frac{1}{\varepsilon}\int_0^t R^Y_{s,s+\varepsilon} (X_{s+\varepsilon} - X_s)ds=0,
\end{equation}
in probability for each $t \in [0,T]$.
$Y'$ is called {\bf stochastic Gubinelli derivative}.
\item $\mathcal{D}_X$ will denote  the couples of
processes $(Y,Y')$ satisfying (\ref{lineardep}) and (\ref{orth}).
\item If $Y$ is an $\R^n$-valued process
whose components are $Y^1,\ldots,Y^n$, then
 $Y$ is said to be stochastically controlled by $X$
if every component $Y^i$ is stochastically controlled by $X$.
The matrix $Y'$ whose rows are stochastic  Gubinelli derivatives $(Y^i)'$ of $Y^i$
is called (matrix) stochastic Gubinelli derivative of $Y$.
The relations \eqref{lineardep} and \eqref{orth} also
make sense in the vector setting.
 $\mathcal{D}_X(\R^n)$ will denote  the couples $(Y,Y')$,
 where $Y$ is a  $\R^n$-valued process,
 being stochastically controlled by $X$ and $Y'$ is a
Gubinelli derivative.
We remark that $R^Y$ also depends on the process $X$.
\end{enumerate}
\end{definition}
Similarly to
the theory of (deterministic) controlled rough paths, in general,
$Y$ can admit different  stochastic Gubinelli derivatives.
 However Proposition \ref{stochex1} states sufficient conditions for uniqueness.

Let us now provide some examples of stochastically controlled processes.
\begin{example} \label{E177}
  Let $X$ be an $\R^d$-valued continuous
  process having all its mutual covariations.
   Let $Y$ be an $\R$-valued process such that, for every $1 \le i \le d$,
 $[Y, X^i]$ exists in the strong sense and  $[Y, X^i] = 0$.
 Consider for instance the three following particular cases.
\begin{itemize}
\item $(Y, X^1,\ldots,X^d)$
  has all its mutual covariations and $[Y,X] = 0$.
  In this case, for every $1 \le i \le d$, $[Y, X^i]$
  exists in the strong sense.
\item 
  Let $Y$ (resp. $X$) be a $\gamma'$-continuous
(resp. $\gamma$-continuous) process
  with $\gamma + \gamma' > 1$.
 Again $[Y,X^i]$ admits its mutual covariations in the strong sense
 and $[Y,X^i] = 0,$
for every $ 1 \le i \le d$
 since
  $$ \int_0^T \vert Y_{s+\varepsilon} - Y_s\vert
  \vert X^i_{s+\varepsilon} - X^i_s\vert \vert \frac{ds}{\varepsilon}
  \le {\rm const} \
  \varepsilon^{\gamma + \gamma' -1} \rightarrow 0,$$
  when  $\varepsilon \rightarrow 0_+$,
  for every $ 1 \le i \le d.$

We recall that, under those conditions, the Young integral
  $\int_0 ^t Y d^{(y)} X,  t \in [0,T]$ exists, see
\cite{young, ber2}.

\item If $X^i, 1 \le i \le d,$ are  continuous bounded variation processes and
  $Y$ is a.s. locally bounded.
\end{itemize}
\begin{enumerate}
\item We claim that $Y$ is stochastically controlled by $X$
 with $Y' \equiv 0$.
\item
 If moreover $[X,X]^\R \equiv 0$, then $Y'$
 can be any locally bounded process:
 therefore
the stochastic Gubinelli derivative is not unique.
 \end{enumerate}
 Indeed, for $ 0 \le s \le t \le T$, write
 $$ Y_t - Y_s = Y'_s (X_t - X_s) + R^Y_{s,t}.$$
 \begin{enumerate}
\item
If $Y' \equiv 0$
we have
$$ \frac{1}{\varepsilon} \int_0^t R_{s,s+\varepsilon} (X_{s+\varepsilon} - X_s) ds \rightarrow 0,$$
when $\varepsilon \rightarrow 0_+$, since
\begin{equation} \label{ERemYoung}
  \frac{1}{\varepsilon}  \int_0^t (Y_{s+\varepsilon} - Y_s) (X_{s+\varepsilon} - X_s) ds
\rightarrow [Y,X]=0,
\end{equation}
when $\varepsilon \rightarrow 0$.
\item If  $[X,X]^\R \equiv 0 $ and $Y'$ is a locally bounded process, then
  we also have
$$ \lim_{\varepsilon \rightarrow 0_+}
 \frac{1}{\varepsilon} \int_0^t  Y'_s  (X_{s+\varepsilon} - X_s)
 (X_{s+\varepsilon} - X_s) ds
  = 0, \ t \in [0,T].$$
This follows by
$$ \lim_{\varepsilon \rightarrow 0_+}
 \frac{1}{\varepsilon} \int_0^t  \vert Y'_s \vert\; \vert X_{s+\varepsilon} - X_s\vert^2  ds
  = 0, \ t \in [0,T],$$
$[X, X]^{\mathbb{R}}=0$ and  Kunita-Watanabe inequality, see e.g.
Proposition 1 4) of \cite{russo2007elements}.
We leave the detailed proof to the reader.
The result follows by \eqref{ERemYoung}.
  \end{enumerate}
\end{example}
In the  second example we show that
a weakly controlled process in the sense of Gubinelli is a stochastically
controlled process.

\begin{example}\label{GubexampleEx}
  Let $X$ be an $\R^d$-valued $\gamma$-H\"older continuous process,
  with $\frac{1}{3} < \gamma < \frac{1}{2}$.
  Let $Y$ be a  $\gamma$-H\"older continuous real-valued process
 such that
  there exists an  $\R^d$-valued process $Y'$,
  so that the remainder term $R^Y$, given through the relation
$$Y_{s,t} = Y'_s (X_{t} - X_s) + R^Y_{s, t},$$
belongs to $C^{[2\gamma]}([0,T]^2).$
In particular $\omega$-a.s., $Y$ is weakly controlled by $X$.
Then, $Y$ is stochastically  controlled by $X$.
 Indeed a.s.
\begin{equation}\label{Gubremainder}
\sup_{0\le t\le T}\Bigg|\frac{1}{\varepsilon}\int_0^t R^Y_{s,s+\varepsilon} (X_{s+\varepsilon} - X_s)ds\Bigg| \le  T \|R\|_{2\gamma}\| X\|_{\gamma} \ \varepsilon^{3\gamma-1}\rightarrow 0,
\end{equation}

as $\varepsilon \rightarrow 0^+$.
In particular the result follows because $\gamma > \frac{1}{3}$.
\end{example}
\begin{example} \label{E179}
  Let $X$ be an  $d$-dimensional continuous semimartingale.
  Let $Z = (Z^1,\ldots,Z^d)$
  where the components $Z^1,\ldots,Z^d$ are càglàd progressively
  measurable processes.
We set
$$ Y_t = \int_0^t Z_s \cdot dX_s := \sum_{i=1}^d \int_0^t Z_s^i dX^i_s, \
t \in [0,T]. $$
Then, the real-valued process $Y$ is stochastically controlled by $X$ and
$Z$ is a Gubinelli stochastic derivative.

Indeed, for $s, t \in [0,T] $ such that $s \le t$, we define $R^Y$
implicitly by the relation
$$ Y_t - Y_s = Z_s^\top (X_t - X_s) + R^Y_{s,t}. $$
We have
$$ \frac{1}{\varepsilon} \int_0^t  R^Y_{s,s+\varepsilon} (X_{s+\varepsilon} - X_s) ds = I_1(t,\varepsilon) -
I_2(t,\varepsilon),$$
with
\begin{eqnarray} \label{E17CovSemBis}
 I_1(t,\varepsilon) &=& \frac{1}{\varepsilon}   \int_0^t  (Y_{s+\varepsilon} - Y_s)  (X_{s+\varepsilon} - X_s) ds  \nonumber \\
I_2(t,\varepsilon) &=&
 \frac{1}{\varepsilon} \int_0^t  Z_s^\top (X_{s+\varepsilon} - X_s) (X_{s+\varepsilon} - X_s)  ds.
\end{eqnarray}
$ I_1(t,\varepsilon)$ converges in probability to
\begin{equation} \label{E17CovSem}
 [Y,X]_t = \int_0^t Z_s^\top d [X,X]_s, \ t \in [0,T],
\end{equation}
by Proposition 9 of \cite{russo2007elements}.
We emphasize that the $k$-component
of the integral on the right-hand side of \eqref{E17CovSemBis}
is
$$ \frac{1}{\varepsilon}
\sum_{j= 1}^d \int_0^t Z^j_s (X^j_{s+\varepsilon} - X^j_s)
(X^k_{s+\varepsilon} - X^k_s) ds.
$$
Reasoning component by component,
it can be also shown
by  Proposition \ref{RC2}.
that $I_2(t,\varepsilon)$ also  converges in probability
to the right-hand side of \eqref{E17CovSem}.
\end{example}
\begin{example}\label{fbmcontrolex}
  Let $X$ be an $d$-dimensional process whose components are finite
  strong cubic variation processes
and at least one component has a zero
 cubic variation.
  Let $f\in C^2(\mathbb{R}^d)$. Then $Y= f(X)$ is
  a stochastically controlled process by $X$ with stochastic
  Gubinelli derivative $Y' = (\nabla f)^\top(X)$.

We prove the result for $d= 1$, leaving to the reader
the general case.
Let $ \omega \in \Omega$ be fixed,
but underlying.
Let $ 0 \le s \le t \le T$.
  Then, Taylor's formula yields
  $$ f(X_t) - f(X_s) = f'(X_s)(X_t-X_s) +
  R^Y_{s,t},$$
where
$$  R^Y_{s,t} =
(X_t-X_s)^2 \int_0^1   f''(X_s + a(X_t - X_s)) (1-a) da.$$
\begin{eqnarray*}
 \left \vert  \frac{1}{\varepsilon}\int_0^t R^Y_{s,s+\varepsilon}
  (X_{s+\varepsilon} - X_s) ds \right \vert &\le&
\sup_{\xi \in I(\omega)}
\vert f''(\xi)
\vert       \int_0^t \vert X_{s+\varepsilon} - X_s\vert^3 \frac{ds}{\varepsilon},
\end{eqnarray*}
where
$$ I(\omega) = [-\min_{t\in[0,T]} X_t(\omega), \max_{t\in[0,T]}  X_t(\omega)].$$
Since the integral on the right-hand side converges in probability (even ucp) to zero,  $R^Y$
fulfills \eqref{orth}.
\end{example}

When $X$ is an $(\shf_t)$-local martingale,  Proposition \ref{stochex1} below shows that somehow
a process $Y$ is stochastically controlled if and only if $Y$ is an $(\shf_t)$-weak Dirichlet process.

\begin{proposition} \label{stochex1}
  Let $X=M $ be an $\mathbb{R}^d$-valued continuous $(\shf_t)$-local martingale.
  Let $Y$ be an $\R$-valued continuous adapted process.
  \begin{enumerate}
    \item Suppose that $Y$ is a weak Dirichlet process.
 Then  $Y$ is stochastically controlled by $M$.
\item Suppose that $Y$ is stochastically controlled by $M$
  and the stochastic Gubinelli derivative $Y'$ is progressively
measurable and càglàd.
  Then $Y$ is a weak Dirichlet process with
  decomposition $ Y = M^Y + A^Y$
  where
  $$ M^Y_t = \int_0^t Y'_s dM_s, t \in [0,T]$$
  and $A^Y$ is an $(\shf_t)$-martingale orthogonal process.
\item (Uniqueness). There is at most one
 stochastic Gubinelli's derivative $Y'$
   in the class of càglàd progressively measurable processes,
  w.r.t to the
 Dol\'eans measure $\mu_{[X]} (d\omega, dt) := d[X,X]^\R_t(\omega) \otimes
 dP(\omega).$

\end{enumerate}
\end{proposition}
\begin{proof}
  \  For simplicity we  suppose  that  $ d= 1$.
  \begin{enumerate}
 \item
  Suppose that $Y$ is a weak Dirichlet process
with canonical decomposition
$$Y =  M^Y + A^Y,$$
where $M^Y$ is the local martingale  and $A^Y$ such that
$A^Y_0 = 0$, is a predictable process such that $[A^Y, N]=0$ for every continuous local martingale $N$. By Galtchouk-Kunita-Watanabe decomposition, see \cite{kw, galtchouk},
 there exist $Z$ and $O$ such that
$$M^Y_t = Y_0 + \int_0^t Z_s dM_s + O_t, \quad  t\in [0,T].  $$
Moreover $O$ is a continuous local martingale such that
$[O,M] = 0$.
Then,
$$Y_t = Y_0 + \int_0^t Z_s dM_s + O_t + A^Y_t, t\in [0,T].$$
We set $Y' := Z$.
Hence,
$$ Y_t - Y_s = Y'_s (M_{t} - M_s) + R^Y_{s,t},$$
where we set
\begin{equation} \label{ERy}
 R^Y_{s,t} = \int_s^t (Z_r - Z_s)^\top dM_r + O_{t} - O_s
 + A^Y_{t} -   A^Y_{s}.
\end{equation}
Condition (\ref{orth}) follows by Remark \ref{EDirCont}
and the fact that $[O,M] = [A^Y,M] = 0.$
\item
Suppose now that $Y$ is stochastically controlled by $M$ with càglàd
stochastic Gubinelli derivative $Y'$. Then, there
is $R^Y$ such that \eqref{lineardep} and \eqref{orth} hold.
Setting $t = s+ \varepsilon$, we have
\begin{equation} \label{EYContr}
 Y_{s+\varepsilon} - Y_s = \int_s^{s+\varepsilon} Y'_r dM_r
+  \int_s^{s+\varepsilon} (Y'_s - Y'_r) dM_r +  R^Y_{s,s+\varepsilon}.
\end{equation}
where $R^Y$ fulfills  \eqref{orth}.
We have
\begin{equation} \label{EYContr1}
 Y_{s+\varepsilon} - Y_s = \int_s^{s+\varepsilon} Y'_r dM_r
 +  \tilde R^Y_{s,s+\varepsilon},
\end{equation}
where
$$\tilde R^Y_{s,s+\epsilon}= \int_s^{s+\varepsilon} (Y'_s - Y'_r) dM_r +  R^Y_{s,s+\varepsilon}, $$
fulfills  \eqref{orth} by Remark \ref{EDirCont}.

Let $N$ be a continuous local martingale.
Multiplying \eqref{EYContr1} by $N_{s+\varepsilon} - N_s$, integrating from
$0$ to $t$, dividing by $\varepsilon$, using \eqref{orth} and
by Proposition 9 of \cite{russo2007elements},
 going to the limit, gives
$$ [Y,N]_t = \int_0^t Y'_r d[M,N]_r, \quad t \in [0,T].$$
This obviously implies that
$Y$ is a weak Dirichlet process with martingale component
$M^Y = Y_0 + \int_0^\cdot Y'_r dM_r.$
\item We discuss now the uniqueness  of the stochastic Gubinelli derivative.
Given two decompositions of $Y$, taking the difference, we reduce the problem to
the following.  Let $Y'$ be a
  càglàd process and $R^Y$, such
  that \eqref{orth} holds for $Y=0$, i.e.
  for every $0 \le s < t \le T$
  \begin{equation}\label{lineardep1}
0 = Y'_s (M_t - M_s) + R^Y_{s ,t},
\end{equation}
satisfies
\begin{equation}\label{orth1}
\lim_{\varepsilon\rightarrow 0^+}\frac{1}{\varepsilon}\int_0^t R^Y_{s,s+\varepsilon}
(M_{s+\varepsilon} - M_s) ds=0,
\end{equation}
in probability for each $t \in [0,T]$.
We need to show that $Y'$ vanishes.
Setting $t = s +\varepsilon$ in \eqref{lineardep1},
 multiplying both sides
by $M_{s+\varepsilon} - M_s$  integrating, for every $ t \in [0,T]$,
taking into account \eqref{orth1}
we get
$$ \lim_{\varepsilon\rightarrow 0^+}  \int_0^t Y'_s
(M_{s+\varepsilon} - M_s)^2 ds=0,$$
in probability.  According to Remark \ref{RC2},
 the left-hand side of previous expression
equals (the limit even holds ucp)
$$ \int_0^\cdot Y'_s d[M,M]_s \equiv 0.$$
This concludes the uniqueness result.
\end{enumerate}
\end{proof}
\begin{remark} \label{EDirCont}
It is not difficult to prove the following.
Let $X$ be an $\R^d$-valued continuous semimartingale with
 canonical decomposition
$X = M + V$.
Let $Z$ be a process in $\shl^2(d[M,M])$.
Then
$$\lim_{\varepsilon \rightarrow 0}
\frac{1}{\varepsilon} \int_0^\cdot ds
\left(\int_s^{s+\varepsilon} (Z_r - Z_s) dX_r\right)(X_{s+\varepsilon} - X_s)
= 0 $$
 ucp.
\end{remark}
The result below partially extends Proposition \ref{stochex1}.

\begin{proposition}\label{stochex2}
  Let $X=M + V $ be an $\mathbb{R}^d$-valued $(\shf_t)$-continuous
  semimartingale,
 where $M$ is a continuous local martingale and $V$ is a
 bounded variation process
vanishing at zero.
 Let $Y$ be a real-valued weak Dirichlet process
$$Y = M^Y + A^Y,$$
where $M^Y$ is the continuous local martingale component
and $A^Y$ is a $(\shf_t)$-martingale orthogonal process vanishing at zero.
Then the following holds.
\begin{enumerate}
  \item
$Y$ is stochastically controlled by $X$.
\item If $Y'$ is a càglàd stochastic Gubinelli's derivative
  then
    \begin{equation} \label{EGubStrat}
[Y,X]_t = \int_0^t Y'_s d[X,X]_s
    \end{equation}
\end{enumerate}
\end{proposition}

\begin{proof} \
By Galtchouk-Kunita-Watanabe decomposition, there exist $Z$ and $O$ such that

$$M^Y_t = Y_0 +  \int_0^t Z_s dM_s + O_t,  t \in [0,T],$$
where $Z\in \shl^2(d[M,M])$, $O$ is a continuous local martingale such that $[O,M] = 0$.
We recall that the space $\shl^2(d[M,M]])$ was defined
at \eqref{SIVR32progVect}.
 Then,
$$Y_t = Y_0 + \int_0^t Z_s dM_s + O_t + A^Y_t, t\in [0,T].$$
Hence,
\begin{equation} \label{E310}
  Y_t - Y_s =  Y'_s (X_t - X_s) + R^Y_{s,t},
  \end{equation}
where we set $Y' = Z, R^Y_{s,t} = \int_s^t (Z_r-Z_s)dM_r
+ O_{s,t} + A^Y_{s,t}$.

Now we recall
\begin{equation} \label{ERecallDir}
  [O,M] = [A^Y,M] = 0.
  \end{equation}
  Taking into account Remark \ref{EDirCont},
  \eqref{E310} and
 \eqref{ERecallDir} show
 condition (\ref{orth1}), which implies (1).

 Then, by Remark \ref{RC2} we have
  $$ \lim_{\varepsilon \rightarrow 0} \int_0^t(Y_{s+\varepsilon} - Y_s)
    \frac{X_{s+\varepsilon}-X_s}{\varepsilon} ds =
    \int_0^t Y'_s d[X,X]_s,$$
  so that (2) is established.
\end{proof}
An interesting consequence of Proposition \ref{stochex2}
is given below.
\begin{corollary} \label{C178} Every continuous $(\shf_t)$-weak Dirichlet process
is stochastically controlled by any $(\shf_t)$-continuous
 semimartingale.
 \end{corollary}

\section{The second order  process and rough integral via regularization}

\label{SSecond}

In the rough paths theory,
given a driving integrator function $X$,
in order to perform integration, one
needs a supplementary ingredient,
often called {\bf second order integral} or
improperly called {\bf L\'evy area}, generally denoted by $\X$.
The couple $\mathbf{X} = (X, \X)$ is often
called {\bf enhanced rough path}.

In our  setup, we are given,
an $\mathbb{R}^d$-valued continuous stochastic process $X$,
which is our reference.
We introduce a stochastic analogue of the  second order integral
in the form
of an
 $\M^{d \times d}$-valued
  random field $\mathbb{X} = (\mathbb{X}_{s,t})$,
indexed by
$[0,T]^2$, vanishing on the diagonal.
$\X$ will be   called {\bf second-order process}.
For $s \le t$,
$\mathbb{X}_{s,t}$
represents formally a double (stochastic) integral
$\int_s^t (X_r - X_s) \otimes dX_r$, which has to be properly defined.
By symmetry, $\mathbb{X}$ can  be extended to $[0,T]^2$, setting, for $s \ge t$,
$$ \mathbb{X}_{s,t} := \mathbb{X}_{t,s}.$$
The pair $\mathbf{X} = (X,\mathbb{X})$
is called {\bf stochastically enhanced process}.

\begin{remark} \label{RChen}
\begin{enumerate}
\item
In the classical rough paths framework,
if $X$ is a deterministic $\gamma$-H\"older continuous path
with $ \frac{1}{3} < \gamma <  \frac{1}{2}$,
$\mathbb{X}$ is supposed to belong to
 $ C^{[2 \gamma]}([0,T]^2)$
and to fulfill
 the so called
 {\bf Chen's relation}  below.
\begin{equation}\label{chen}
  -\mathbb{X}_{u,t} + \mathbb{X}_{s,t}  - \mathbb{X}_{s,u} = (X_{u} - X_s)
   (X_{t} - X_u)^\top, u,s,t \in [0,T].
\end{equation}

\item
In the literature
one often introduces a decomposition
of $ \mathbb{X}$ into  a symmetric and  an antisymmetric component, i.e.
\begin{eqnarray*}
\text{sym}(\mathbb{X}_{s,t})(i,j) &:=&   \frac{1}{2}\Big(\mathbb{X}_{s,t}(i,j)  +
          \mathbb{X}_{s,t}(j,i) \Big) \\
\text{anti}(\mathbb{X}_{s,t})(i,j) &:=&
\frac{1}{2}\Big(\mathbb{X}_{s,t}(i,j)  -
                                        \mathbb{X}_{s,t}(j,i) \Big),
\end{eqnarray*}
$ 1\le i,j\le d,$
so that
\begin{equation}\label{Esym_antisym}
  \mathbb{X}_{s,t} =
\text{sym}(\mathbb{X}_{s,t}) + \text{anti}(\mathbb{X}_{s,t}).
\end{equation}
\item
We say that the pair $\mathbf{X} = (X,\mathbb{X})$ is \textbf{geometric}
 if
$$ {\rm sym} (\mathbb{X}_{st}) = \frac{1}{2}(X_t- X_s)  (X_t-X_s)^\top, \quad
 s,t\in [0,T].$$

\end{enumerate}

\end{remark}

A typical second-order process $\X$
is defined setting
\begin{equation} \label{RExBold}
  \mathbb{X}_{s,t}:=\int_s^t (X_r-X_s)\otimes d^\circ X_r,
 \end{equation}
 provided that previous definite
 symmetric  integral
 exists, for every $0 \le s \le t \le T$,
see \eqref{SIVR1MatrixSymm}.

We can also consider another $\X$, replacing
the symmetric integral with the forward integral,
 i.e.
\begin{equation} \label{RExBoldForw}
 \mathbb{X}_{s,t}:=\int_s^t (X_r-X_s)\otimes d^- X_r,
\end{equation}
provided that previous definite forward integrals exist,
 exists, for every $(s,t) \in [0,T]^2, 0 \le s \le t \le T,0$
see \eqref{SIVR1Matrix}.
\begin{example} \label{ESemimartingale}
  Let $X$ be an $\R^d$-valued continuous semimartingale.
Then, for $1 \le i,j \le d$, one often considers
 $$ \mathbb{X}^{\rm stra}_{s,t}(i,j):=
\left(\int_s^t (X_r-X_s)\otimes d^\circ X_r\right)(i,j)
= \int_s^t (X^i_r-X^i_s) \circ d X^j_r$$
and
$$ \mathbb{X}^{\rm ito}_{s,t}(i,j):=
\left(\int_s^t (X_r-X_s)\otimes d^- {X_r}\right)(i,j)
= \int_s^t (X^i_r-X^i_s) dX^j_r,$$
where the integrals in the right-hand side  are respectively  intended in the
Stratonovich and It\^o sense.
\end{example}

\section{Rough stochastic integration via regularizations}

\label{S74}

In this section we still consider
our $\mathbb{R}^d$-valued reference process $X$, equipped with its second-order process $\mathbb{X}$.
Inspired by \cite{gubinelli2004}, we start with the definition of the integral.

\begin{definition}\label{roughstochDEF}
 A couple $(Y,Y')\in \mathcal{D}_X$ is \textbf{rough stochastically integrable}
 if
\begin{equation}\label{stochroughint}
\int_0^t Y_s d\mathbf{X}_s:=\lim_{\varepsilon\rightarrow 0}\frac{1}{\varepsilon}\int_0^t \Big( Y_s X^\top_{s,s+\varepsilon} + Y'_s\mathbb{X}_{s,s+\varepsilon}\Big)ds
\end{equation}
exists in probability for each $t\in [0,T]$. Previous integral is
called {\bf rough stochastic integral}
and it is a row vector.
 \index{rough stochastic integral}
\end{definition}

We remark that if $Y' = 0$ the rough stochastic integral coincides
with the  forward integral $\int_0^t Y d^-X, t \in [0,T]$.
 In previous definition, we make an abuse of notation: we omit the dependence of the integral on $Y'$ which in general affects the limit but it is usually clear from the context.

We introduce now a backward version of $\int_0^\cdot Y d\mathbf{X}$,
i.e. the {\bf backward rough integral}

$$\int_0^t Y_s \stackrel{\rm \leftarrow }{d\mathbf{X}}_s:=
\lim_{\epsilon\rightarrow 0^+}\frac{1}{\epsilon}\int_0^t \Big(Y_{s+\epsilon}
X^\top_{s,s+\epsilon} + Y'_{s+\epsilon} \mathbb{X}_{s,s+\epsilon}\Big) ds,$$
in probability for  $(Y,Y')\in \mathcal{D}_X$.
Previous expression is again a row vector.
\begin{remark} \label{TReversal}
  Given an $\R^n$-valued process $(Y_{t \in [0,T]})$,
  we denote $\hat Y_t:= Y_{T-t}, \ t \in [0,T]$.
  \begin{enumerate}
  \item The introduction of the backward rough integral is
    justified by the following observation.
  By an easy change of variables $s \mapsto T-s$
  we easily show that, for every $t \in [0,T]$,
  \begin{equation} \label{ERoughTReversal}
    \int_0^t Y_s \stackrel{\rm \leftarrow }{d\mathbf{X}}_s =
    - \int_{T-t}^T  {\widehat Y}_s  d\mathbf{\widehat X}_s.
    \end{equation}
This holds of course with the convention that
$\hat Y$ is equipped with $\hat Y'$ as Gubinelli derivative.
\item \eqref{ERoughTReversal} is reminiscent of a well-known
  property which states that
  $$ \int_0^t Y d^+X = - \int_{T-t}^T \widehat Y d^- \widehat X,$$
  where the left-hand side is the {\bf backward integral}
   $ \int_0^t Y d^+X,$
    see Proposition 1 3), see \cite{russo2007elements}. 
\end{enumerate}
\end{remark}

Let us give a simple example which connects deterministic regularization
 approach with rough paths.

\begin{proposition}\label{existenceGUB}
Let $\mathbf{X} = (X,\mathbb{X})$ be an a.s.  enhanced
 rough path, where
a.s. $ X\in C^{[\gamma]}([0,T])$ with $\frac{1}{3} < \gamma < \frac{1}{2}$.
We suppose that a.s. $\X \in C^{[2 \gamma]}([0,T]^2)$  and
it fulfills the Chen's relation.
Let  $Y$ be a process such that a.s. its paths are weakly controlled
in the sense of
Definition \ref{Gubexample} with Gubinelli derivative $Y'$.
The following properties  hold.

\begin{enumerate}
\item The limit
$$  \lim_{\varepsilon\rightarrow 0}\frac{1}{\varepsilon}\int_0^\cdot
\Big( Y_s X_{s,s+\varepsilon}^\top + Y'_s\mathbb{X}_{s,s+\varepsilon}\Big)ds,$$
exists uniformly on $[0,T]$ and it coincides a.s.
with the {\bf Gubinelli integral}.
\index{Gubinelli integral}
In particular, (\ref{stochroughint}) exists.
\item The limit
\begin{equation}\label{EApprBackw}
  \lim_{\varepsilon\rightarrow 0}\frac{1}{\varepsilon}\int_0^\cdot
  \Big( Y_{s+\varepsilon} X^\top_{s,s+\varepsilon} + Y'_{s+\varepsilon} \mathbb{X}_{s,s+\varepsilon}\Big)ds,
  \end{equation}

exists uniformly on $[0,T]$ a.s. and it coincides a.s.
with the rough Gubinelli integral
 as described in \cite{gubinelli2004}.
\item The rough stochastic integrals
$\int_0^\cdot Y_s {d\mathbf{X}}_s $ and
$\int_0^\cdot Y_s \stackrel{\rm \leftarrow }{d\mathbf{X}}_s$
exist and they are equal a.s. to the Gubinelli integral.
\end{enumerate}
\end{proposition}
\begin{remark} \label{R1716}
When $Y$ is $\gamma'$-H\"older continuous and $X$ is
  $\gamma$-H\"older continuous, with $\gamma + \gamma' > 1$,
Proposition 3. in Section 2.2 of \cite{russo2007elements}
stated that the Young integral $\int_0^t Y d^{(y)} X,$
 equals both
the forward and backward integrals $\int_0^t Y d^{\mp} X$.
Proposition \ref{existenceGUB} states an analogous theorem for
the Gubinelli integral, which equals both
$\int_0^\cdot Y_s {d\mathbf{X}}_s $ and
$\int_0^\cdot Y_s \stackrel{\rm \leftarrow }{d\mathbf{X}}_s$.

 \end{remark}

We introduce now  the notion of multi-increments.
Let $k \in \{1,2, 3\}.$
We denote by $\shc_k$ the space of continuous functions
$g:[0,T]^k \rightarrow \R$, denoted by
$(t_1, \ldots, t_k) \mapsto g_{t_1, \ldots, t_k}$  such that
 $g_{t_1, \ldots, t_k}  = 0$ whenever $t_i = t_{i+1}$ for some
 $1\le i \le k-1$.

 For $g \in \shc_2$, we have defined
 $\Vert g \Vert_\alpha$ at \eqref{EIncrBis}.
 For $g\in \mathcal{C}_3$, we set
$$\|g\|_{\alpha,\beta}:= \sup_{s,u,t\in [0,T]}\frac{|g_{tus}|}{|u-s|^\alpha |t-s|^\beta},$$

$$\|g\|_{\mu}:= \inf\Big\{ \sum_i \|g^i\|_{\rho_i,\mu-\rho_i};
g = \sum_{i} g^i, 0 < \rho_i < \mu \Big\},$$
where the latter infimum is taken over all sequences $\{g^i\in \mathcal{C}_3\}$
such that $g = \sum_i g^i$ and for all choices of $\rho_i \in ]0,\mu[$.
 We say that $g \in C^\mu([0,T]^3)$ if $\|g\|_\mu < \infty$.

We introduce the maps
\begin{enumerate}
\item $\delta_1: \shc_1 \rightarrow \shc_2$
defined by
$(\delta_1 f)_{s,t} = f(t) - f(s)$.
\item $\delta_2: \shc_2 \rightarrow \shc_3$
 defined by
$$\delta_2 f_{t_1,t_2,t_3} = - f_{t_2,t_3} + f_{t_1,t_3} -  f_{t_1,t_2}.$$
\end{enumerate}
If $k = 1,2$  and
$f \in \shc_k$, $\delta_k f$ is called
{\bf $k$-increment}  of the function $f$.

In the proof of Proposition \ref{existenceGUB},
as
in \cite{gubtindel},
 it is crucial to make use of the so called
 {\bf Sewing Lemma}.
The lemma below follows directly from Proposition 2.3
in \cite{gubtindel}.

\begin{lemma}  \label{LemmaSewing}
Let $g \in \shc_2$   such that $\delta_2 g \in C^{[\mu]}([0,T]^3),
$
for some $\mu >1$. Then, there exists a unique (up to a constant)
$I \in \shc_1$
and $\shr \in C^{[\mu]}([0,T]^2)$
 such that
$$ g = \delta_1 I + \shr.$$
\end{lemma}
\begin{prooff} \ (of Proposition \ref{existenceGUB}).
\begin{enumerate}
\item
We set
 \begin{equation}\label{1germ}
   A_{s,t} = Y_s (X_t-X_s)^\top + Y'_s \mathbb{X}_{s,t}, \quad (s,t) \in [0,T]^2.
 \end{equation}
 Then the $2$-increment of $A$ is given by
 \begin{eqnarray}
   \nonumber(\delta_2 A)_{t_1,t_2,t_3} &=&
         Y_{t_1} (X_{t_3} - X_{t_1})^\top + Y'_{t_1}  \X_{t_1,t_3} \\
\nonumber  &-& Y_{t_2} (X_{t_3} - X_{t_2})^\top - Y'_{t_2}  \X_{t_2,t_3}
         - Y_{t_1} (X_{t_2} - X_{t_1})^\top - Y'_{t_1}  \X_{t_1,t_2} \\
   \nonumber
&= &  (Y_{t_2} - Y_{t_1}) (X_{t_2} - X_{t_3})^\top
+ Y'_{t_1} (\X_{t_1,t_3} - \X_{t_2,t_3} - \X_{t_1,t_2}) \\ \nonumber
&-& (Y'_{t_2} -  Y'_{t_1}) \X_{t_2,t_3} \\ \nonumber
&=&  \Big\{ Y_{t_2} -Y_{t_1} - Y'_{t_1}
 \big(X_{t_2} - X_{t_1} \big) \Big\}
 \big(X_{t_2} - X_{t_3}\big)^\top + \big(\delta_1 Y'\big)_{t_1t_2}\mathbb{X}_{t_2,t_3}\\
 \label{difgerm}& &\\
   \nonumber &=& R^Y_{t_1, t_2} (\delta_1 X)^\top_{t_2,t_3} +
                 \big(\delta_1 Y'\big)_{t_1,t_2}\mathbb{X}_{t_2, t_3},
 \end{eqnarray}
 where the third equality follows by Chen's relation.
 By Definition \ref{Gubexample} we have a.s. $Y' \in  C^{[\gamma]}([0,T]),
 R^Y \in C^{[2\gamma]}([0,T]^2)$ and we also have
$\X \in C^{[2\gamma]}([0,T]^2)$.
 Consequently
 $\delta_2 A \in C^{[3\gamma]}([0,T]^3)$.

 Then, setting $\mu = 3 \gamma$, outside a null set,  Lemma \ref{LemmaSewing}
 applied  to $g = A$,
 provides
  an unique (up to a constant)
  a continuous process $I$ such that
 $$ A_{s,s+\varepsilon} = I_{s+\varepsilon} - I_s + \shr_{s,s+\varepsilon},$$
 where $\shr \in C^{[3\gamma]}([0,T]^2)$.
For a given $\varepsilon>0$ and $t\in [0,T]$, we then have
 $$\frac{1}{\varepsilon}\int_0^t A_{s,s+\varepsilon}ds= \frac{1}{\varepsilon}\int_0^t I_{s,s+\varepsilon}ds + \frac{1}{\varepsilon}\int_0^t \shr_{s,s+\varepsilon}ds$$
and
 \begin{equation}\label{tel}
 \lim_{\varepsilon\rightarrow 0^+}\frac{1}{\varepsilon}\int_0^\cdot I_{s,s+\varepsilon}ds
= I_\cdot - I_0,
 \end{equation}
 uniformly in $[0,T]$. By using the fact that
$\shr \in C^{[3 \gamma]}([0,T]^2)$.
 we have
 $$\frac{1}{\varepsilon}\sup_{s\in [0,T]}|\shr_{s,s+\varepsilon}|\le \frac{\varepsilon^{3\gamma}}{\varepsilon}\|\shr\|_{3\gamma}\rightarrow 0, $$
 as $\varepsilon\downarrow 0$. This completes the proof.
\item We fix $\omega$.
  The quantity \eqref{EApprBackw} converges to $I_t$
  where  $I$ is again the (unique) function appearing
in the Sewing Lemma
\ref{LemmaSewing}.
  The arguments are similar to those of item 1.
\item This is a direct consequence of previous points
  and the fact that a.s. $I$ also coincides with the Gubinelli integral.

\end{enumerate}
 \end{prooff}

\begin{theorem} \label{RoughSem}
   Let  $X=(X_t)_{t \in [0,T]}$ be a given continuous $(\shf_t)$-semimartingale
 with values in $\R^d$ and $Y$ be an
  $(\shf_t)$-weak Dirichlet process.
  We set
  $ \X:= \X^{\rm stra}, $ see Example \ref{ESemimartingale}.

\noindent  Then the rough
stochastic integral of $Y$ (with càglàd progressively measurable, stochastic Gubinelli derivative
$Y'$)
 with respect to $\textbf{X}=(X, \X)$ coincides with
the Stratonovich integral
i.e.
  \begin{align} \label{eq:roughstoch=fiskstratonovich}
  \int_0^\cdot Y_s d\textbf{X}_s = \int_0^\cdot Y_s \circ dX_s.
  \end{align}
    \end{theorem}
  \begin{remark} \label{RStocRougSem}
    \begin{enumerate}
      \item
 In Proposition \ref{stochex2} we have shown the
    existence of a progressively measurable process $Y'$
    such that $(Y, Y')$ belongs to $\shd_X$.
  \item  \eqref{eq:roughstoch=fiskstratonovich} implies that
the value of the rough stochastic integral
    does not depend on $Y'$.
\end{enumerate}
\end{remark}

  \begin{prooff} \ (of Theorem \ref{RoughSem}).

    The rough stochastic integral $\int_0^\cdot Y d \textbf{X}_s$
    defined in
    \eqref{stochroughint}
    exists if we prove in particular that
 the two limits below
\begin{equation} \label{E1718}
  \lim_{\varepsilon \rightarrow 0}  \frac{1}{\varepsilon}
  \int_0^t Y_s X^\top_{s, s+\varepsilon} ds  \quad \text{ and }  \lim_{\varepsilon \rightarrow 0}
\frac{1}{\varepsilon} \int_0^t Y'_s \X_{s,s+\varepsilon} ds,
  \end{equation}
exist in probability.
  We will even prove the ucp convergence of \eqref{E1718}.
  Let us fix $i \in \{1, \dots,d\}$.
 By Proposition 6. in \cite{russo2007elements} we have
  \begin{align} \label{eq:proofroughstochintegral-1}
  \displaystyle \lim_{\varepsilon \rightarrow 0} \int_0^t Y_s \frac{X^i_{s+\varepsilon}- X^i_s}{\varepsilon} ds= \int_0^t Y_s d^{-}X^i_s= \int_0^t Y_s dX^i_s,
  \end{align}
 ucp, where the second integral in the equality is the usual It\^o's
 stochastic integral.

  \noindent We show now that
  \begin{align}\label{eq:proofroughstochintegral-2}
    \frac{1}{\varepsilon} \int_0^t Y'_s \X_{s,s+\varepsilon} ds
    \rightarrow \frac{1}{2}[Y,X]_t, t \in [0,T],
  \end{align}
   holds ucp as $\varepsilon \rightarrow 0$.

  \noindent Let $i \in \{1, \dots, d\}$.
We write, for every $t \in [0,T],$ an element of  the vector
  $  \frac{1}{\varepsilon} \int_0^t Y'_s \X_{s,s+\varepsilon} ds$ as
  \begin{align*}
    \frac{1}{\varepsilon} \Big (\int_0^t Y'_s \X_{s, s+\varepsilon} ds
    \Big )_{i} &=
                 \sum_{k=1}^d                                                           \int_0^t  (Y'_s)^k
                \left( \frac{1}{\varepsilon}\int_s^{s+\varepsilon} (X^k_r - X^k_s)
                 \circ d X^i_r\right) ds.
  \end{align*}
  The definition of Stratonovich integral yields
  \[\sum_{k=1}^d\int_0^t   (Y'_s)^k
    \Big(\frac{1}{\varepsilon} \int_s^{s+\varepsilon} (X^k_r - X^k_s) \circ
 d X^i_r\Big) ds =\sum_{k=1}^d  \int_0^t (Y'_s)^k
           \Big(\frac{1}{\varepsilon}  \int_s^{s+\varepsilon}
 (X^k_r- X^k_s)  d X^i_r\Big) ds\]
 \[\qquad \qquad \qquad \qquad \qquad \qquad \qquad \qquad +\frac{1}{2 \varepsilon} \sum_{k=1}^d \int_0^t  (Y'_s)^k [X^k - X^k_s,
 X^i]_{s,s+\varepsilon} ds.\]
  Obviously $[X^k - X^k_s, X^i] = [X^k,X^i]$.
  Since the covariations  $[X^k,X^i]$ are bounded variation processes,
item 7. of Proposition 1. in \cite{russo2007elements}
 shows that
  the second term in the right-hand side of the latter
 identity converges  in ucp as $\varepsilon \rightarrow 0$ to
  \begin{align*}
    \frac{1}{2} \sum_{k=1}^d \int_0^t (Y'_r)^k  d[X^k, X^i]_r = \frac{1}{2}
\left(\int_0^t Y'_r d[X]_r\right)_i = \frac{1}{2} [Y,X^i]_t,
  \end{align*}
  where the latter equality follows by \eqref{EGubStrat} in
  Proposition \ref{stochex2}.
 
  \noindent We complete the proof if we show
 that for every $i \in \{1, \dots, d\}$ and $k \in \{1, \dots,d\}$  the  ucp limit
  \begin{align}  \label{eq:thmroughstochintegral-lastlimittoprove1}
  \int_0^t (Y'_s)^k \left(\frac{1}{\varepsilon} \int_s^{s+\varepsilon} (X^k_r-X^k_s) dX^i_r\right) ds \rightarrow 0 \quad \text{ as } \varepsilon \rightarrow 0,
  \end{align}
holds. Let $M^i +V^i$ be the canonical decomposition of
the semimartingale $X^i$. By usual localization arguments
we can reduce to the case when $[M^i], \Vert V^i \Vert(T),
X^i, (Y')$ are bounded processes.
 Using the stochastic Fubini's Theorem
 (see Theorem 64, Chapter 6 in \cite{prot}),
 we can write
  \begin{align*}
    \int_0^t (Y'_s)^k
  \left (\frac{1}{\varepsilon}   \int_s^{s+\varepsilon} (X^k_r - X^k_s)
dX^i_r\right)
   ds = \int_0^{t+\varepsilon}  \left(\frac{1}{\varepsilon} \int_{(r-\varepsilon)^+}^{r \wedge t} (Y'_s)^k (X^k_r - X^k_s) ds\right) dX^i_r.
  \end{align*}
  \noindent  For $\varepsilon>0,$  and $k \in\{1,\ldots,d\},$
let us define the auxiliary process
  \begin{align*}
  \xi^\varepsilon(t):=  \frac{1}{\varepsilon} \int_{(r-\varepsilon)^+}^{r \wedge t} (Y'_s)^k (X^k_r - X^k_s) ds.
\end{align*}
Controlling the border terms as usual,
by Problem 5.25 Chapter 1. of \cite{ks}
(\ref{eq:thmroughstochintegral-lastlimittoprove1}), it remains to show that
 the limit in probability
\begin{align} \label{eq:thmroughstochintegral-lastlimittoprove2}
\int_0^{T} |\xi^\varepsilon(r)|^2 d[X^i]_r \rightarrow 0 \text{ as } \varepsilon \rightarrow 0 \quad \text{holds}.
\end{align}
\noindent
 Denoting by $\delta(X,\cdot)$ the continuity modulus of $X$
on $[0,T]$,
$$
\int_0^{T} |\xi^\varepsilon(r)|^2 d[X^i]_r
\le \delta(X, \varepsilon)^2  \sup_{s \in [0,T]} \vert (Y'_s)^{k}\vert^2
[X^i]_T,
$$
which obviously converges a.s. to zero.
This concludes
 the proof of (\ref{eq:proofroughstochintegral-2}).

\noindent Combining (\ref{eq:proofroughstochintegral-1}) and (\ref{eq:proofroughstochintegral-2}) we finish the proof of (\ref{eq:roughstoch=fiskstratonovich}).
  \end{prooff}

  Through a similar but simpler proof (left to the reader) than the one of
Theorem \ref{RoughSem}
  we have the following.
  \begin{theorem} \label{RoughSemIto}
   Let  $X=(X_t)_{t \in [0,T]}$ be a given continuous $(\shf_t)$-semimartingale
   with values in $\R^d$ and let $Y$
   be  a.s.
   bounded and progressively measurable. Suppose moreover
   that $Y$ has a càglàd progressively measurable
   Gubinelli derivative $Y'$.
  We set
  $ \X:= \X^{\rm ito}, $ see Example \ref{ESemimartingale}.
 Then the rough
stochastic integral of $Y$ with respect to $\textbf{X}=(X, \X)$ coincides with
the It\^o integral of $Y$ with respect to $X$,
i.e.
  \begin{align} \label{eq:roughstoch=ito}
  \int_0^\cdot Y_s d\textbf{X}_s = \int_0^\cdot Y_s  dX_s.
  \end{align}
    \end{theorem}
    Theorems \ref{RoughSem} and \ref{RoughSemIto} somehow extend
Proposition 5.1 in \cite{friz} and Corollary 5.2 in  \cite{hairerbook}.
In this paper, $(Y,Y')$ does not necessarily have H\"older continuous paths
with the classical regularity in the sense of rough paths.

\bigskip

{\bf ACKNOWLEDGMENTS.}

The support of A. Ohashi and F. Russo research related to this paper
 was financially supported by the Regional Program MATH-AmSud 2018,
 project Stochastic analysis of non-Markovian phenomena (NMARKOVSOC),
 grant 88887.197425/2018-00.
The authors are also grateful to Pierre Vallois (Nancy)
for stimulating discussions.


\bibliographystyle{plain}
\begin{quote}
\bibliography{../../../BIBLIO_FILE/biblio-PhD-Alan.bib}

\def\cprime{$'$}
\begin{thebibliography}{10}

\bibitem{nunno}
D.~R. {Ba{\~n}os}, F.~{Cordoni}, G.~{Di Nunno}, L.~{Di Persio}, and E.~E.
  {R{\o}se}.
\newblock {Stochastic systems with memory and jumps}.
\newblock {\em ArXiv e-prints}, March 2016.

\bibitem{BandiniRusso1}
E.~Bandini and F.~Russo.
\newblock Weak {D}irichlet processes with jumps.
\newblock {\em Stochastic Process. Appl.}, 127(12):4139--4189, 2017.

\bibitem{BandiniRusso2}
E.~Bandini and F.~Russo.
\newblock Special weak {D}irichlet processes and {BSDE}s driven by a random
  measure.
\newblock {\em Bernoulli}, 24(4A):2569--2609, 2018.

\bibitem{bbv11}
B.~B\'erard-Bergery and P.~Vallois.
\newblock Convergence at first and second order of some approximations of
  stochastic integrals.
\newblock In {\em S\'eminaire de {P}robabilit\'es {XLIII}}, volume 2006 of {\em
  Lecture Notes in Math.}, pages 241--268. Springer, Berlin, 2011.

\bibitem{ber2}
J.~Bertoin.
\newblock Sur une int\'egrale pour les processus \`a {$\alpha$}-variation
  born\'ee.
\newblock {\em Ann. Probab.}, 17(4):1521--1535, 1989.

\bibitem{chevyrev2019}
I.~Chevyrev and P.~K. Friz.
\newblock Canonical rdes and general semimartingales as rough paths.
\newblock {\em Annals of Probability}, 47(1):420--463, 2019.

\bibitem{cjms}
F.~Coquet, A.~Jakubowski, J.~M{\'e}min, and L.~S{\l}omi{\'n}ski.
\newblock Natural decomposition of processes and weak {D}irichlet processes.
\newblock In {\em In memoriam {P}aul-{A}ndr\'e {M}eyer: {S}\'eminaire de
  {P}robabilit\'es {XXXIX}}, volume 1874 of {\em Lecture Notes in Math.}, pages
  81--116. Springer, Berlin, 2006.

\bibitem{cosso_russo15a}
A.~Cosso and F.~Russo.
\newblock Functional and {B}anach space stochastic calculi: path-dependent
  {K}olmogorov equations associated with the frame of a {B}rownian motion.
\newblock In {\em Stochastics of environmental and financial
  economics---{C}entre of {A}dvanced {S}tudy, {O}slo, {N}orway, 2014--2015},
  volume 138 of {\em Springer Proc. Math. Stat.}, pages 27--80. Springer, Cham,
  2016.

\bibitem{coutin2005semi}
L.~Coutin and A.~Lejay.
\newblock Semi-martingales and rough paths theory.
\newblock {\em Electronic Journal of Probability}, 10:761--785, 2005.

\bibitem{coviello2bis}
R.~Coviello, C.~Di~Girolami, and F.~Russo.
\newblock On stochastic calculus related to financial assets without
  semimartingales.
\newblock {\em Bull. Sci. Math.}, 135(6-7):733--774, 2011.

\bibitem{coviello1}
R.~Coviello and F.~Russo.
\newblock Nonsemimartingales: stochastic differential equations and weak
  {D}irichlet processes.
\newblock {\em Ann. Probab.}, 35(1):255--308, 2007.

\bibitem{DGR}
C.~Di~Girolami and F.~Russo.
\newblock Infinite dimensional stochastic calculus via regularization and
  applications.
\newblock {\em Preprint \textup{HAL-INRIA, inria-00473947 version 1}},
  (Unpublished), 2010.

\bibitem{DGRnote}
C.~Di~Girolami and F.~Russo.
\newblock {C}lark-{O}cone type formula for non-semimartingales with finite
  quadratic variation.
\newblock {\em Comptes Rendus Mathematique}, 349(3-4):209 -- 214, 2011.

\bibitem{DGR2}
C.~Di~Girolami and F.~Russo.
\newblock Generalized covariation and extended {Fu}kushima decompositions for
  {B}anach space valued processes. application to windows of {D}irichlet
  processes.
\newblock {\em Infinite Dimensional Analysis, Quantum Probability and Related
  Topics (IDA-QP).}, 15(2), 6 2012.

\bibitem{DGR1}
C.~Di~Girolami and Francesco Russo.
\newblock {G}eneralized covariation for {B}anach space valued processes,
  {I}t\^o formula and applications.
\newblock {\em Osaka Journal of Mathematics}, 51:729--783, 2014.

\bibitem{DGRclassical}
C.~Di~Girolami and Francesco Russo.
\newblock About classical solutions of the path-dependent heat equation.
\newblock {\em Random Oper. Stoch. Equ.}, 28(1):35--62, 2020.

\bibitem{er}
M.~Errami and F.~Russo.
\newblock Covariation de convolution de martingales.
\newblock {\em C. R. Acad. Sci. Paris S\'er. I Math.}, 326(5):601--606, 1998.

\bibitem{er2}
M.~Errami and F.~Russo.
\newblock {$n$}-covariation, generalized {D}irichlet processes and calculus
  with respect to finite cubic variation processes.
\newblock {\em Stochastic Process. Appl.}, 104(2):259--299, 2003.

\bibitem{flaguru}
F.~Flandoli, M.~Gubinelli, and F.~Russo.
\newblock On the regularity of stochastic currents, fractional {B}rownian
  motion and applications to a turbulence model.
\newblock {\em Ann. Inst. Henri Poincar\'e Probab. Stat.}, 45(2):545--576,
  2009.

\bibitem{FolDir}
H.~F{\"o}llmer.
\newblock Dirichlet processes.
\newblock In {\em Stochastic integrals ({P}roc. {S}ympos., {U}niv. {D}urham,
  {D}urham, 1980)}, volume 851 of {\em Lecture Notes in Math.}, pages 476--478.
  Springer, Berlin, 1981.

\bibitem{friz2020rough}
P.~Friz and P.~Zorin-Kranich.
\newblock Rough semimartingales and $ p $-variation estimates for martingale
  transforms.
\newblock {\em ArXiv preprint arXiv:2008.08897}, 2020.

\bibitem{hairerbook}
P.~K. Friz and M.~Hairer.
\newblock {\em A course on rough paths}.
\newblock Universitext. Springer, Cham, 2014.
\newblock With an introduction to regularity structures.

\bibitem{friz}
P.~K. Friz and N.~B. Victoir.
\newblock {\em Multidimensional stochastic processes as rough paths}, volume
  120 of {\em Cambridge Studies in Advanced Mathematics}.
\newblock Cambridge University Press, Cambridge, 2010.
\newblock Theory and applications.

\bibitem{galtchouk}
L.~I. Gal{\cprime}{\v{c}}uk.
\newblock A representation of certain martingales.
\newblock {\em Teor. Verojatnost. i Primenen.}, 21(3):613--620, 1976.

\bibitem{gr1}
F.~Gozzi and F.~Russo.
\newblock Verification theorems for stochastic optimal control problems via a
  time dependent {F}ukushima-{D}irichlet decomposition.
\newblock {\em Stochastic Process. Appl.}, 116(11):1530--1562, 2006.

\bibitem{gr}
F.~Gozzi and F.~Russo.
\newblock Weak {D}irichlet processes with a stochastic control perspective.
\newblock {\em Stochastic Process. Appl.}, 116(11):1563--1583, 2006.

\bibitem{gradno}
M.~Gradinaru and I.~Nourdin.
\newblock Approximation at first and second order of {$m$}-order integrals of
  the fractional {B}rownian motion and of certain semimartingales.
\newblock {\em Electron. J. Probab.}, 8:no. 18, 26 pp. (electronic), 2003.

\bibitem{gnrv}
M.~Gradinaru, Ivan Nourdin, F.~Russo, and P.~Vallois.
\newblock {$m$}-order integrals and generalized {I}t\^o's formula: the case of
  a fractional {B}rownian motion with any {H}urst index.
\newblock {\em Ann. Inst. H. Poincar\'e Probab. Statist.}, 41(4):781--806,
  2005.

\bibitem{grv}
M.~Gradinaru, F.~Russo, and P.~Vallois.
\newblock Generalized covariations, local time and {S}tratonovich {I}t\^o's
  formula for fractional {B}rownian motion with {H}urst index {$H\ge\frac14$}.
\newblock {\em Ann. Probab.}, 31(4):1772--1820, 2003.

\bibitem{gubinelli2004}
M.~Gubinelli.
\newblock Controlling rough paths.
\newblock {\em Journal of Functional Analysis}, 216(1):86--140, 2004.

\bibitem{gubtindel}
M.~Gubinelli and S.~Tindel.
\newblock Rough evolution equations.
\newblock {\em The Annals of Probability}, 38(1):1--75, 2010.

\bibitem{ks}
I.~Karatzas and S.~E. Shreve.
\newblock {\em Brownian motion and stochastic calculus}, volume 113 of {\em
  Graduate Texts in Mathematics}.
\newblock Springer-Verlag, New York, second edition, 1991.

\bibitem{kw}
H.~Kunita and S.~Watanabe.
\newblock On square integrable martingales.
\newblock {\em Nagoya Math. J.}, 30:209--245, 1967.

\bibitem{le2020stochastic}
Khoa L{\^e}.
\newblock A stochastic sewing lemma and applications.
\newblock {\em Electronic Journal of Probability}, 25, 2020.

\bibitem{tindelliu}
Y.~Liu, Z.~Selk, and S.~Tindel.
\newblock Convergence of trapezoid rule to rough integrals.
\newblock {\em arXiv preprint arXiv:2005.06500}, 2020.

\bibitem{lyonsq}
T.~Lyons and Zh. Qian.
\newblock {\em System control and rough paths}.
\newblock Oxford Mathematical Monographs. Oxford University Press, Oxford,
  2002.
\newblock Oxford Science Publications.

\bibitem{lyons}
T.~J. Lyons.
\newblock Differential equations driven by rough signals.
\newblock {\em Rev. Mat. Iberoamericana}, 14(2):215--310, 1998.

\bibitem{ORrough}
A.~Ohashi and F.~Russo.
\newblock Rough paths, skorohod integrals driven by covariance singular
  gaussian processes.
\newblock {\em In preparation}, 2021.

\bibitem{prot}
Ph. Protter.
\newblock {\em Stochastic integration and differential equations}, volume~21 of
  {\em Applications of Mathematics (New York)}.
\newblock Springer-Verlag, Berlin, 1990.
\newblock A new approach.

\bibitem{RVCras}
F.~Russo and P.~Vallois.
\newblock Int\'egrales progressive, r\'etrograde et sym\'etrique de processus
  non adapt\'es.
\newblock {\em C. R. Acad. Sci. Paris S\'er. I Math.}, 312(8):615--618, 1991.

\bibitem{rv1}
F.~Russo and P.~Vallois.
\newblock Forward, backward and symmetric stochastic integration.
\newblock {\em Probab. Theory Related Fields}, 97(3):403--421, 1993.

\bibitem{rv93}
F.~Russo and P.~Vallois.
\newblock Noncausal stochastic integration for l\`ad l\`ag processes.
\newblock In {\em Stochastic analysis and related topics (Oslo, 1992)},
  volume~8 of {\em Stochastics Monogr.}, pages 227--263. Gordon and Breach,
  Montreux, 1993.

\bibitem{rv2}
F.~Russo and P.~Vallois.
\newblock The generalized covariation process and {I}t\^o formula.
\newblock {\em Stochastic Process. Appl.}, 59(1):81--104, 1995.

\bibitem{rv96}
F.~Russo and P.~Vallois.
\newblock It\^o formula for {$C\sp 1$}-functions of semimartingales.
\newblock {\em Probab. Theory Related Fields}, 104(1):27--41, 1996.

\bibitem{rv4}
F.~Russo and P.~Vallois.
\newblock Stochastic calculus with respect to continuous finite quadratic
  variation processes.
\newblock {\em Stochastics Stochastics Rep.}, 70(1-2):1--40, 2000.

\bibitem{russo2007elements}
F.~Russo and P.~Vallois.
\newblock Elements of stochastic calculus via regularization.
\newblock In {\em S{\'e}minaire de Probabilit{\'e}s XL}, pages 147--185.
  Springer, 2007.

\bibitem{russoviens}
F.~Russo and F.~Viens.
\newblock Gaussian and non-{G}aussian processes of zero power variation.
\newblock {\em ESAIM Probab. Stat.}, 19:414--439, 2015.

\bibitem{young}
L.~C. Young.
\newblock An inequality of {H}\"older type, connected with {S}tieltjes
  integration.
\newblock {\em Acta Math.}, 67:251--282, 1936.

\end{thebibliography}
 \end{quote}

\end{document}